\documentclass[12pt]{article}
\usepackage[colorlinks,linkcolor=Blue,citecolor=Blue]{hyperref}
\usepackage[margin=1in]{geometry}  
\usepackage{graphicx}              
\usepackage{amsmath}               
\usepackage{amsfonts}              
\usepackage{amsthm}                
\usepackage{amsmath}
\usepackage{mathtools}
\usepackage{times}
\usepackage{array, amssymb, amsmath, graphics}
\usepackage{enumerate}
\usepackage{epsfig}
\usepackage{float}

\newtheorem{theorem}{Theorem}

\newtheorem{problem}{Problem}
\newtheorem{lemma}{Lemma}
\newtheorem{proposition}{Proposition}
\newtheorem{corollary}{Corollary}

\newtheorem{conjecture}{Conjecture}

\newtheorem{pretheorema}{{\bf Theorem}}

\newenvironment{theorema}[1]{\begin{pretheorema}
{\hspace{-0.2em}{\rm #1}{\bf}}}{\end{pretheorema}}
\newtheorem{precor}{{\bf Corollary}}

\newtheorem{prelemlem}{{\bf Lemma}}

\newcommand{\gal}[1]{{\rm gal}(#1)}
\newcommand{\arb}[1]{{\rm arb}(#1)}
\newcommand{\sa}[1]{{\rm sa}(#1)}
\usepackage[usenames,dvipsnames]{pstricks}
\usepackage{epsfig}
\usepackage{pst-grad} 
\usepackage{pst-plot} 
\title{\bf On the star arboricity of hypercubes\thanks{This research is partially supported by a grant from the INSF.}
}
\author{{\sc Negin  Karisani\footnote{ karisani@alum.sharif.edu,  emahmood@sharif.edu,  n.sobhanikhakestar@gmail.com.}, E.S. Mahmoodian},
\\ {\sc Narges K. Sobhani}
 }
\date{}
\begin{document}
\maketitle
\vspace*{-1cm}
\begin{center}
\footnotesize{
Department of Mathematical Sciences \\
 Sharif University of Technology \\
 P.O. Box 11155--9415 \\
 Tehran, I.R. IRAN \\
}
 \end{center}
%
\begin{abstract}
A Hypercube $Q_n$ is a graph  in which the vertices are all binary vectors of length $n$,
and two vertices are adjacent if and only if their components differ in exactly one place.
A  galaxy or a star forest is a union of vertex disjoint stars. The star arboricity of a graph $G$, $\sa{G}$, is the minimum  number of galaxies which partition the edge set of $G$. In this paper among other results, we determine the exact values of $\sa{Q_n}$ for $n \in \{ 2^k-3, 2^k+1, 2^k+2, 2^i+2^j-4\}$, $i \geq j \geq 2$. We also improve the last known upper bound of $\sa{Q_n}$ and show the relation between $\sa{G}$ and square coloring.
\\
%
\end{abstract}
%
\section{Introduction and preliminaries}       
Hypercubes have numerous applications in computer science such as studying networks. Their  architecture has played an important role in the development of parallel processing and is still quite popular and influential \cite{parhami2002}. An \textit{$n$-cube} or \textit{$n$-dimensional hypercube}, $Q_n$, is a graph  in which the vertices are all binary vectors of length $n$,
and two vertices are adjacent if and only if the Hamming distance between
them is $1$, i.e. their components differ in $1$ place.  $Q_n$ is also defined recursively in terms of the cartesian product of two graphs as follows:
$$\begin{array}{l}
Q_1=K_2 \\
Q_n= Q_{n-1}  \square K_2,
\end{array}$$
where $\square$ stands for the cartesian product.

A  \textit{galaxy}  or a \textit{star forest} is  a  union  of  vertex  disjoint  stars. The \textit{star arboricity}  of a graph $G$, denoted by $\sa{G}$, is the minimum  number of galaxies which partition the edge set of $G$. The study of decomposing graphs into  galaxies is naturally suggested by the analysis of certain communication networks such as radio networks. As an example, suppose that we need to transmit once along every edge, in order to check that there is indeed a connection between each adjacent pair. It is explained in  \cite{MR1273589} that the minimum number of steps in which we can finish all the required transmissions is precisely $\sa{G}$. 
Star arboricity was introduced  by Akiyama and Kano in $1982$ \cite{MR778395}. They called it star decomposition index. In other literature some authors have used concepts and notations such as galactic number,  $\gal{G}$, and star number, ${\rm st}(G)$ or ${\rm s}(G)$. 

Star arboricity is closely related to \textit{arboricity}, the minimum number of forests which partition the edges of a graph $G$  and is denoted by $\arb{G}$. But unlike arboricity which is easy, even determining whether the star arboricity of an arbitrary  graph   is at  most $2$, is shown to be an NP-complete  problem \cite{MR959906}, also see \cite{MR2528046}. Clearly, by  definition $\arb{G} \leq \sa{G}$.  Furthermore,  it  is  easy to  see  that  any  tree  can  be  covered  by  two  star  forests,  thus $\sa{G} \leq 2\arb{G}$. Alon et al. \cite{MR1194728} showed that for each  $k$,  there exists a  graph  $G_k$  with  $\arb{G_k}=k$ and $\sa{G_k}=2k$. They also showed that for any graph $G$,  $\sa{G} \leq \arb{G}+O(\log_2 \Delta(G))$, where $\Delta(G)$ is the maximum degree of $G$.

In \cite{MR959906} and \cite{MR885263} the star arboricity of $Q_n$ is studied and it is shown that $\sa{Q_{2^n-2}}=2^{n-1}$ and $\sa{Q_{2^n-1}}=2^{n-1}+1$. Here by extending earlier results, we find exact values of $\sa{Q_n}$ for $n \in \{2^k-4, 2^k-3, 2^k, 2^k+1, 2^k+2, 2^k+4, 2^k+2^j-4\}$. Also we introduce a new upper bound and show a relation between $\sa{G}$ and square coloring. 
%
\section{Some earlier results}
In this section we mention some earlier results about the star arboricity of general graphs which are used in the next section.

In the following theorem Akiyama and Kano found an exact value for the star arboricity of complete graphs $K_n$.
\begin{theorema}{ \rm (\cite{MR778395})} \label{tcomplete}
Let $n \geq 4$. Then the star arboricity of the complete graph of order $n$ is $\lceil{n \over 2} \rceil+1$, {\rm i.e.}
$
\sa{K_n}=\lceil {n \over 2} \rceil+1.
$
\end{theorema}
In the  next  lemma  an upper bound  for  the star arboricity of  product  of  two graphs is given.
\begin{lemma}{ \rm (\cite{MR959906})} \label{lcart}
The star arboricity of the cartesian product of two graphs satisfies \linebreak
$
\sa{G \square H} \leq \sa{G}+\sa{H}.
$
\end{lemma}
Next we state some of Truszczy{\'n}ski's results \cite{MR885263} which will be used in this paper.
\begin{theorema}{ \rm (\cite{MR885263})}
\label{Tdreg}
Let $G$ be an $n$-regular graph, $n \geq 2$. Then
$
\sa{G} \geq \lceil {n \over 2} \rceil +1.
$
\end{theorema}
\begin{lemma}{ \rm (\cite{MR885263})} \label{l6}
Let  $G$ be an $n$-regular graph, where $n$ is an even number. If \ $\chi(G)>~{{n \over 2}+1}$ \ or if \ ${{n \over 2}+1}$ does not divide $|V(G)|$, then
$\sa{G} \geq  {{n \over 2}+2}$. 
\end{lemma}
The following question is also raised about the upper bound for $\sa{G}$.
\begin{problem} { \rm (\cite{MR885263})}
Is it true that for every $n$-regular graph $G$,
$$
\Big\lceil {n \over 2} \Big\rceil +1 \leq \sa{G} \leq \Big\lceil {n \over 2} \Big\rceil +2 \ ?
$$
\end{problem}
\begin{lemma}{ \rm (\cite{MR885263})} \label{ltruz}
If $k \geq 2$, then there is a partition ${\cal A}=\{ A_1,A_2, \ldots , A_{2^{k-1}} \}$ \ of \ $V(Q_{2^k-2})$ such that
\begin{enumerate}[(i)]
\item 
for every $i$, $1\leq i \leq 2^{k-1}$, $A_i$ is independent,
\item
for every $i$, $j$, $1\leq i < j\leq 2^{k-1}$, the subgraph of $Q_{2^k-2}$ induced by $A_i \cup A_j$, is $2$-regular.
\end{enumerate}
\end{lemma}
Proof of Lemma~\ref{ltruz} in \cite{MR885263} is constructive and as an example a decomposition of $Q_6$ into $4$ sets is presented  in Table~\ref{tab:Q6}. This will be used in the next section.
\begin{table}[h]
\begin{center}
\includegraphics[scale=.45]{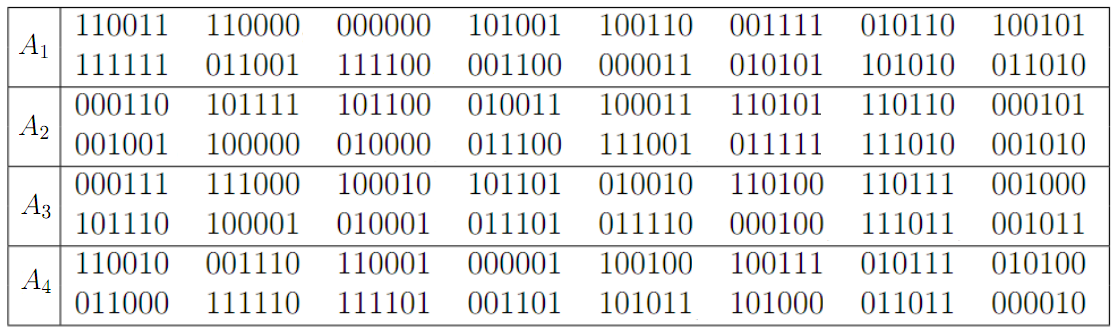}
\caption{A vertex decomposition of $Q_6$ by Lemma~\ref{ltruz}. }\label{tab:Q6}
\end{center}
\end{table}

In the earlier results the star arboricity of $Q_n$ was determined in just two cases. 
\begin{theorema}{ \rm (\cite{MR885263})} \label{t1}
$\sa{Q_{2^k-2}}=2^{k-1}$ for $k \geq 2$.
\end{theorema}
%
%
%
%
By Theorem~\ref{Tdreg},  Lemma~\ref{l6} and  Theorem~\ref{t1} we have,
\begin{corollary} \label{ceven}
$\sa{Q_{n}} \ge \lfloor \frac{n}{2} \rfloor +2$, \ except for $n=2^a-2$, $a\ge 2$.

For $n=2^a-2$, we have $\sa{Q_{2^a-2}}= \lfloor \frac{n}{2} \rfloor +1=2^{a-1}$.
\end{corollary}
\begin{corollary}{ \rm (also in \cite{MR959906})}\label{l2}
$\sa{Q_{2^k-1}}=2^{k-1}+1$, $k \geq 2$.
\end{corollary}
\begin{proof}
$\sa{Q_{2^k-1}} \geq 2^{k-1}+1$ by Corollary~\ref{ceven}, and $\sa{Q_{2^k-1}}\leq \sa{Q_{2^k-2}}+1=2^{k-1}+1$.\end{proof}
%
The following bounds also are given in \cite{MR885263}.
\begin{theorema}{ \rm (\cite{MR885263})} \label{tlog}
$\lceil {n+1 \over 2}\rceil+1 \leq \sa{Q_n} \leq \lceil {n \over 2}\rceil +\log_2 n$, \ for every $n \geq 3$, {\rm [}except for $n=2^{a}-2$ \  and \ $a\geq2$  {\rm ]}.
\end{theorema}
In the next section we introduce more exact values of star arboricity of some $Q_n$. 
\section{Hypercubes}
In  this section we focus on the star arboricity of hypercubes and extend earlier results.

Based on the results we conjecture that,
\begin{conjecture}
\label{qn}
 For \  $n \neq 2^a-2$, $\sa{Q_{n}} = \lfloor \frac{n}{2} \rfloor +2$,  and for \  $n = 2^a-2$, $\sa{Q_{2^a-2}}= \lfloor \frac{n}{2} \rfloor +1=2^{a-1}$, $a \geq 2$.
\end{conjecture}
Note that the second part of the Conjecture~\ref{qn} is known to be true (Theorem~\ref{t1}). 
\begin{theorem}\label{tt}
If Conjecture~\ref{qn} holds for an odd integer $n$, then it holds for $n-1$ and $n+1$.
\end{theorem}
\begin{proof}
If $n-1=2^a-2$ or $n+1=2^a-2$ for some $a$, then the statement follows. Otherwise let $n=2k+1$. For $n-1=2k$, we have $\sa{Q_{2k}} \leq \sa{Q_{2k+1}}=k+2$; the statement follows by Corollary~\ref{ceven}. For $n+1=2k+2$, we have $\sa{Q_{2k+2}} \leq \sa{Q_{2k+1}}+1=k+3$ and  again by  Corollary~\ref{ceven} the statement follows.~\end{proof}
By Theorem~\ref{tt}, one only needs to show Conjecture~\ref{qn}, for odd numbers.
\subsection{Exact values }
\begin{proposition}\label{l4}
$\sa{Q_{2^k+2^{j}-4}}=2^{k-1}+2^{j-1}$, \ for \ $k \geq j \geq 2$.
\end{proposition}
\begin{proof}
By Lemma~\ref{lcart} and Theorem~\ref{t1}, each we have $\sa{Q_{2^k+2^{j}-4}} \le \sa{Q_{2^k-2}}+ \sa{Q_{2^{j}-2}} =2^{k-1}+2^{j-1 }$. Also by Corollary~\ref{ceven}, $\sa{Q_{2^k+2^{j}-4}} \ge 2^{k-1}+2^{j-1}$. So $\sa{Q_{2^k+2^{j}-4}}=2^{k-1}+2^{j-1}$.~\end{proof}
%
%
%
The following lemma is useful tool for the next theorem.
\begin{lemma}\label{lg1}
If a graph $G$ satisfies the following conditions then $\sa{G}=2$.
\begin{enumerate}
\item \label{c1}
G is tripartite with $V(G)=V_1 \cup V_2 \cup V_3$.
\item \label{c2}
Each vertex in $V_1$ or $V_2$ has degree $4$ and each vertex in $V_3$ has degree $2$.
\item \label{c3}
 Each vertex in $V_1$ or $V_2$ has exactly $2$ neighbours in $V_3$ and each vertex in $V_3$ is adjacent to both $V_1$ and $V_2$.
\end{enumerate}
\end{lemma}
\begin{proof}
We decompose the edges of $G$ into two galaxies in such a way that all of the stars are $K_{1,3}$. The induced subgraph on $H=\langle V_1 \cup V_2 \rangle$ is a bipartite graph. This bipartite graph must be a disjoint union of some even cycles. So we can partition edges of $H$ into the sets $M_1$ and $M_2$, such that each of them is a perfect matching in $H$. Now we partition the edges of $G$ into two galaxies $G_1$ and $G_2$: The first one, $G_1$, is the union of $M_1$, with an induced subgraph of $\langle V_1 \cup V_3 \rangle$. In a similar way $G_2$ is the union of $M_2$, with an induced subgraph of $\langle V_2 \cup V_3 \rangle$.
\end{proof}

\begin{theorem} \label{l9}
$\sa{Q_{2^k+1}}=2^{k-1}+2$, \ for \ $k \geq 2$.
\end{theorem}
\begin{proof} By Corollary~\ref{ceven}, $ \sa{Q_{2^k+1}} \geq 2^{k-1}+2$. So it suffices to partition the edges of $Q_{2^k+1}$ into $2^{k-1}+2$ galaxies. 
We know that $Q_{2^k+1}= Q_{2^k-2}  \square Q_3$. Also by Lemma~\ref{ltruz} the vertices of $Q_{2^k-2}$ can be partitioned into $2^{k-1}$ sets, ${\cal A}=\{ A_1,A_2, \ldots , A_{2^{k-1}} \}$, such that an induced subgraph between each two sets is a $2$-regular subgraph. Now we need some conventions and notations. 
For a fixed $3$-bit codeword $c$ from $Q_3$, we extend each set $A_i$, $1 \leq i \leq 2^{k-1}$, to a set $A_i(c)$, which has codewords of length $2^k+1$, such that for each codeword $c^\prime \in A_i$, we append $c$ to the end of $c^\prime$ in $A_i(c)$.
Therefore for each pair of $i$ and $j$, the induced subgraph on $A_i(c) \cup A_j(c)$ in $Q_{2^k+1}$ is  a $2$-regular graph which can be decomposed into two perfect matchings. We denote  them by   $A_i(c) \rightarrow A_j(c)$ and $A_j(c) \rightarrow A_i(c)$. Also for any two $3$-bit codewords $c_1$ and $c_2$ which are different in only one bit, the induced subgraph of two sets  $A_i(c_1)$ and $A_i(c_2)$ is a perfect matching between those sets and we denote it by $A_i(c_1) \parallel A_i(c_2)$. Also for any  $3$-bit codeword $c$ we denote by  $c^i$, the $3$-bit codeword which is different from $c$
 exactly in the $i$-th bit, and by $\overline{c}$, the complement of $c$ that defers with $c$ in all bits. Also the set of all $3$-bit codewords with even weights is denoted by $E_c$,  i.e. $E_c=\{000, 011, 110,101\}$.
 Now we are ready to introduce our $2^{k-1}$ galaxies. For each  $i$, $1 \leq i \leq 2^{k-1}$, define $G_i$ as follows:
$$G_i = \mathop{\bigcup}_{c \in E_c}\{ [\mathop{\cup}_{j \neq i}\big(A_i(c) \rightarrow A_j(c)\big)] \cup [A_i(c) \parallel A_i(c^1)] \cup [\mathop{\cup}_{k \neq i , i+1} \big(A_{i+1}(c^1) \rightarrow A_k(c^1)\big)] \},$$
\ \ \  \ $1 \leq j , k \leq 2^{k-1}.$ \\ \\
Note that in the above formula the indices $i$ and $j$ and $k$ are considered modulo $2^{k-1}$.

Next we prove that the following statements hold for $G_i$:
\begin{description}
\item[] {\bf Statement 1.}\label{a} Every $G_i$ is a galaxy.
\item[] {\bf Statement 2.} \label{b} The remaining edges satisfy conditions of Lemma~\ref{lg1}.
\end{description}
By using these two statements, we derive $\sa{Q_{2^k+1}}\leq 2^{k-1}+2$, therefore the statement of the theorem will be held. 

Before proving these statements, as an example, we illustrate our construction in case of $Q_9$. We have $Q_9 = Q_6 \square Q_3$. Previously in Table~\ref{tab:Q6} a decomposition of $Q_6$ into $4$ sets as in Lemma~\ref{ltruz} is presented. In Figure~\ref{fig:aQ9}:(a), we have shown a $Q_3$ and a figure in which each of these partitioned sets is a vertex of $2K_4$, where each edge stands for a perfect matching between two corresponding sets. In Figure~\ref{fig:aQ9}:(b) the galaxy $G_1$ is represented, where again each edge represents a perfect matching. To illustrate more, we have also shown $G_3$ in Figure~\ref{fig:aQ9}:(c). Figure~\ref{fig:aQ9}:(d) is for the last two galaxies obtained from the remaining edges. Each of these presented galaxies can be mapped to a galaxy of $Q_9$ by a blow up.
\begin{figure}
\begin{center}
\includegraphics{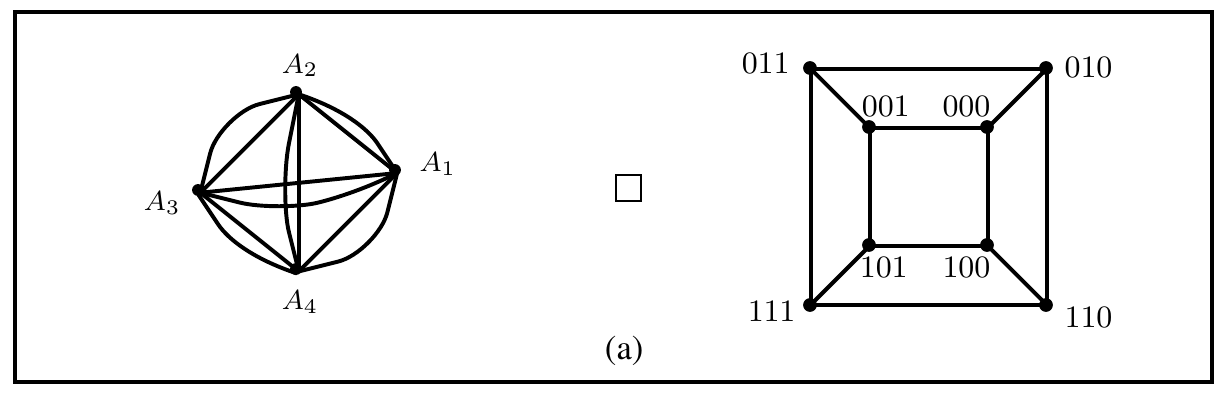}\\ 
\includegraphics{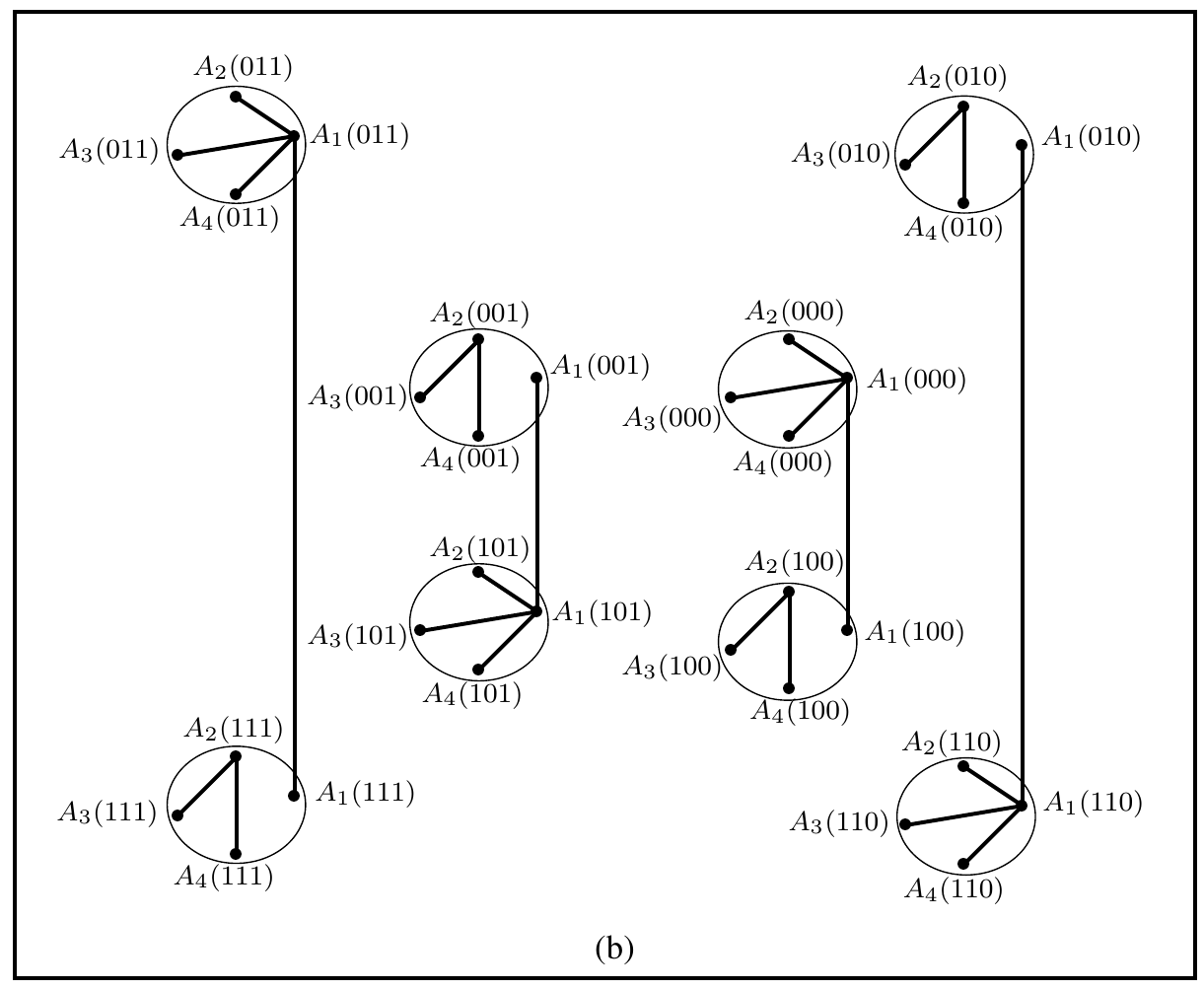}\\
\includegraphics[scale=.7]{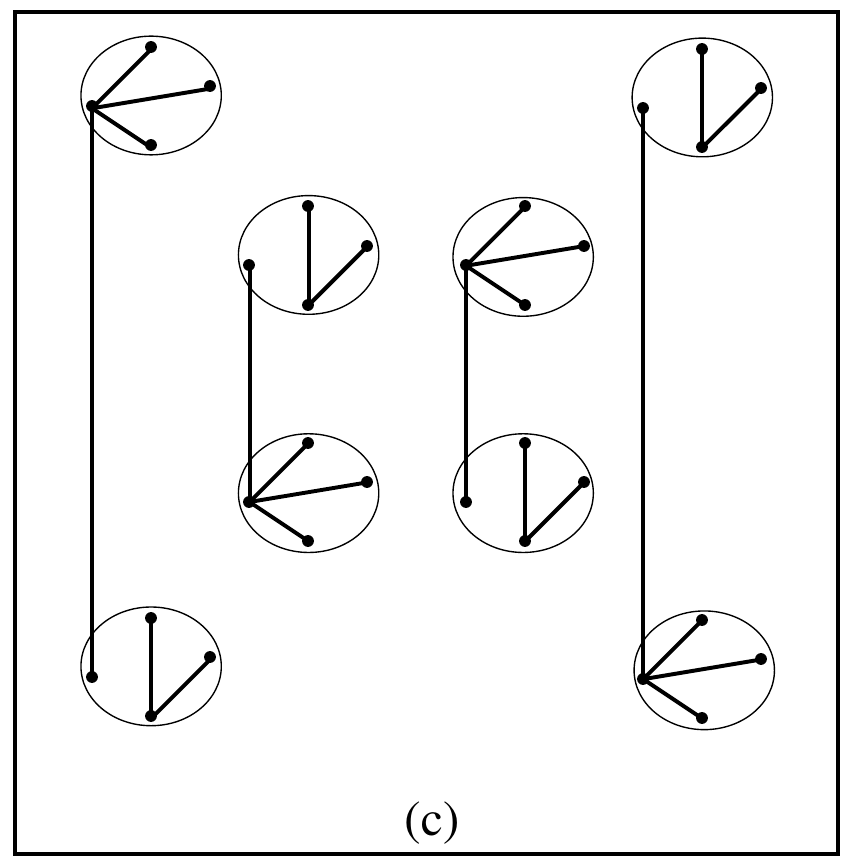} 
\includegraphics[scale=.7]{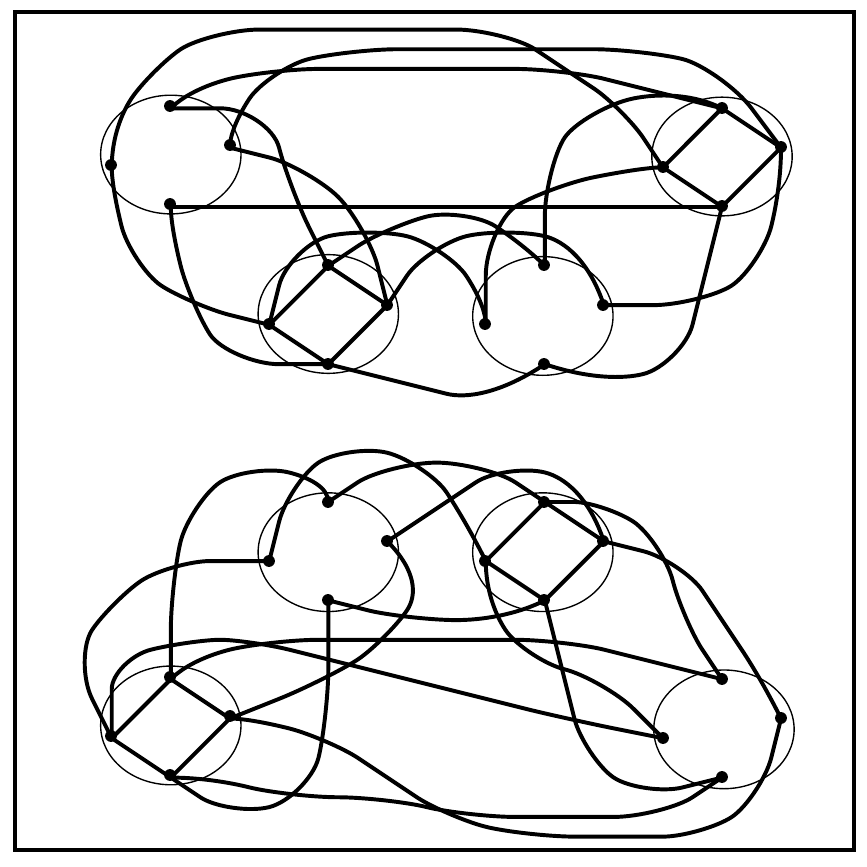}
\caption{(a) $Q_6 \square Q_3$, (b) $G_1$,  (c) $G_3$, (d) Galaxies obtained in Statement 2. }\label{fig:aQ9}
\end{center}
\end{figure}

{ \bf Proof of Statement 1.} 
By the definitions, it is obvious that for each $G_i$ and for each $c \in E_c$, the independent sets $A_i(c)$, $A_j(c)$, $A_i(c^1)$, $A_{i+1}(c^1)$ and $A_k(c^1)$, $1 \leq j, \ k  \leq 2^{k-1}$, \ $j \neq i$,  \ $k \neq i,i+1$, are mutually disjoint.  By construction, every component is a star, and  $[\mathop{\cup}_{j \neq i}\big(A_i(c) \rightarrow A_j(c)\big)] \cup [A_i(c) \parallel A_i(c^1)]$ is a union of stars with centers at the vertices in $A_i(c)$, and similarly $\mathop{\cup}_{k \neq i , i+1} \big(A_{i+1}(c^1) \rightarrow A_k(c^1)\big)$ is a union of stars with centers at the vertices in $A_{i+1}(c^1)$. Since here every $c^1$ corresponds to exactly one $c$, it follows that these stars have no overlaps.

{\bf Proof of Statement 2.} We must prove that the remaining edges make a new graph which satisfies conditions of Lemma~\ref{lg1}. 
%
Let
$$
\begin{array}{l}
V_1=\bigcup_{1 \leq s \leq 2^{k-2}}\big( A_{2s}(001) \cup A_{2s}(100) \cup A_{2s-1}(010) \cup A_{2s-1}(111) \big),  \\
V_2=\bigcup_{1 \leq s \leq 2^{k-2}}\big( A_{2s-1}(001) \cup A_{2s-1}(100) \cup A_{2s}(010) \cup A_{2s}(111) \big),  \\
V_3=\bigcup_{c \in E_c}\big( \cup_{1 \leq i \leq 2^{k-1}} A_i(c) \big).
\end{array}
$$
\\ 
First, we show that in the graph of remaining edges, $V_3$ is a set of vertices with degree $2$. By construction, each vertex in $V_3$, i.e. $A_i(c)$, $c \in E_c$, is a center with degree $2^{k-1}-1+1$ in a star of $G_i$, and is a leaf in any other $G_j$, $j \neq i$,  $1 \leq i,j \leq 2^{k-1}$. So each vertex in $A_i(c)$ is covered in the given galaxies by totally $(2^{k-1}-1+1) + (2^{k-1}-1)\times 1=2^k-1$ adjacent edges. Since $Q_{2^k+1}$, is $(2^k+1)$-regular graph, so in the remaining graph the degree of each vertex in $A_i(c)$ is $2$. Next we show that each of the remaining vertices which are in $V_1 \cup V_2=\bigcup_{c \in E_c}\big(\cup_i A_i(\overline{c})\big)$, $1\leq i \leq 2^{k-1}$, has degree $4$. For each vertex $v$ in  $A_i(\overline{c})$, it is clear that the galaxy $G_{i-1}$ covers $2^{k-1}-2$ edges of $v$ and each of the other galaxies, $G_j$, $j \neq i-1$, covers $1$ edge of $v$. So $(2^{k-1}-2)+(2^{k-1}-1) \times 1=2^k-3$ edges of each vertex in $A_i(\overline{c})$ is covered by all $G_j$, $1\leq j \leq 2^{k-1}$. Thus each vertex in $A_i(\overline{c})$ has $4$ uncovered edges.
Therefore each of the vertices in $V_1 \cup V_2$ has degree $4$ and the vertices in $V_3$ have degree $2$, hence Condition~(\ref{c2}) is satisfied.

For Condition~(\ref{c3}), by the given construction we note that in the first $2^{k-1}$ galaxies each vertex of $V_3$, i.e.  $A_i(c)$ ($i$ fixed and $c \in E_c$) is adjacent just to one of the vertices of  $V_1 \cup V_2$, i.e. $A_i(c^1)$.
So in the remaining graph each vertex in $V_3$, which has degree $2$, is adjacent to both a vertex with degree $4$ in $A_i(c^2)$, and another vertex with degree $4$ in $A_i(c^3)$. Since $A_i(c^2) \cup A_i(c^3)\subseteq V_1 \cup V_2$, so vertices of $V_3$ are independent. Since  $c^2$ and $c^3$ are different in the last two bits, so $A_i(c^2)$ and $A_i(c^3)$ can not be in the same $V_j$, $j \in \{1, 2\}$.
Also each vertex in $A_i(\overline{c})$, which is in $V_1$ or $V_2$, is adjacent to a vertex with degree $2$ in $A_i(\overline{c}^2)$ and another vertex with degree $2$ in $A_i(\overline{c}^3)$.  Thus each vertex in $V_1$ and $V_2$ has exactly $2$ neighbors in $V_3$. Hence Condition~(\ref{c3}) holds.

To prove Condition~(\ref{c1}), it remains to show that each of $V_1$ and $V_2$ is an independent set. 
As we have seen each vertex in $A_i(\overline{c})$ has degree $4$ and two of its neighbors are in $V_3$. The other two neighbors are in the sets $A_{i+1}(\overline{c})$ and $A_{i-1}(\overline{c})$. By the definition of $V_1$ and $V_2$, it is obvious that  $A_i(\overline{c})$ is not in the same $V_j$, $j \in \{1, 2\}$, as $A_{i+1}(\overline{c})$ and $A_{i-1}(\overline{c})$ are. 
\end{proof}

\newpage
\begin{lemma} \label{lf}
$\sa{Q_n}=\lfloor {n \over 2}\rfloor+2$ \ for \ $n=2^k+4$ \ and \ $2^k-4 \leq n \leq 2^k+2$ \ except for \ $2^k-2$.
\end{lemma}
\begin{proof}
We can check that the statement holds for  $n \leq 10$, see Table~\ref{tab:table2}.
\begin{table}[!ht]
\begin{center}
$\begin{array}{c|cccccccccc}
n                &1 &2 &3 & 4 &5 &6 & 7 &8 &9 & 10    \\ \hline
{\rm sa}(Q_{n}) &1 &2 &3 & 4 &4 &4 & 5 &6 &6 & 7 \\ \hline
\end{array}
$
\caption{$\sa{Q_n}$ for $n \leq 10$. }\label{tab:table2}
\end{center}
\end{table}

So let $k \geq 3$, other than previous mentioned cases in Theorem~\ref{t1}, Theorem~\ref{l9} and Corollary~\ref{l2}, for the remaining cases the statement holds as follows:
\begin{itemize}
\item $n=2^k+4$, (by Propsition~\ref{l4} for $j=3$),
\item $n=2^k+2$,  (by Theorem~\ref{l9} and Theorem~\ref{tt} for $n=2^k+1$),
\item $n=2^k$, (by Propsition~\ref{l4} for $j=2$) and (by Theorem~\ref{l9} and Theorem~\ref{tt} for $n=2^k+1$),
\item $n=2^k-3$, ($\sa{Q_{2^k-3}} \leq \sa{Q_{2^k-2}}=2^{k-1}$ and by Corollary~\ref{ceven} for $n=2^k-3$),
\item $n=2^m-4$,  (by Propsition~\ref{l4} for $k=j=m-1$) and (by Theorem~\ref{tt} for $n=2^m-3$).
\end{itemize}
\vspace*{-5mm}\end{proof}
%
The value of $\sa{Q_{2^k+3}}$ is left open, but we know that $2^{k-1}+3 \leq \sa{Q_{2^k+3}} \leq 2^{k-1}+4$. So the smallest unknown case is $\sa{Q_{11}}$.

\subsection{An upper bound}
In the following theorem, we improve the known upper bound on $\sa{Q_n}$ by a method similar
to the proof of  Theorem~\ref{tlog}.
\begin{theorem}\label{tUp}
Let $n$ be an even integer. We can write $n$ as $n=\sum_{j=1}^l(2^{i_j}-2)+r$, where $r$ is in $ {\cal R}= \{ 2^k+2,  2^s+2^t-4\}$, $s \geq t \geq 2$, $i_1 > i_2 > \cdots > i_l$ and $l$ is the smallest number with this property. Also we have
 $$
 \sa{Q_n} \leq  {n \over 2}+l+2.
 $$


\end{theorem}
\begin{proof} 
If $n \in {\cal R}$ then $l=0$ and the statement holds by Lemma~\ref{lf} or by  Proposition~\ref{l4}. Otherwise  it is easy to see that $n$ can be written as $2^{i_1}-2+r_1$, where $i_1$ is the largest possible integer and $r_1$ is the remainder. If $r_1=0$ then the statement holds by Theorem~\ref{t1}, else $4 < r_1 < 2^{i_1}-2$. If $r_1 \in \cal R$ then $l=1$ and by Lemma~\ref{lf} and Proposition~\ref{l4}, $\sa{Q_{r_1}}={r_1 \over 2}+2$ and $\sa{Q_n} \leq \sa{Q_{2^{i_1}-2}}+\sa{Q_{r_1}}=2^{i_1-1}+{r_1 \over 2}+2={n \over 2}+3$, and we are done. Else, again $r_1$ can be written as $2^{i_2}-2+r_2$, where $i_2$ is the largest possible integer and so on. Thus assume $n=\sum_{j=1}^l(2^{i_j}-2)+r$, $r \in {\cal R}$, then 
\\ \\
$
\begin{array}{rl}
\sa{Q_{n}} & \leq \sa{Q_{2^{i_1}-2}} + \sa{Q_{2^{i_2}-2}}+ \cdots + \sa{Q_{2^{i_l}-2}} + \sa{Q_r} \\
& =  2^{i_1-1} + 2^{i_2-1} + \cdots + 2^{i_l-1} + \dfrac{\strut r }{\strut 2}+2\\
& = \dfrac{\strut 2^{i_1}-2 }{2} + \dfrac{\strut 2^{i_2}-2}{ 2} + \cdots + \dfrac{\strut 2^{i_l}-2 }{ 2} + l+ \dfrac{ r }{ 2}+2 \\
& = \dfrac{\strut n-r }{ 2} + l + \dfrac{r }{ 2}+2 \\
& = \dfrac{\strut n }{ 2} + l + 2.
\end{array}$
\\ \vspace*{-9mm}
\end{proof}
\vspace*{9mm}
\begin{corollary}
Let $n \geq 1$, then $\sa{Q_n}\leq \lceil {n \over 2}\rceil +l +2$, where $l$ is obtained for $n$ or $n-1$ as in Theorem~\ref{tUp}, whether $n$ is even or odd, respectively.
\end{corollary}
\begin{corollary}
$\sa{Q_n} \leq \lceil {n \over 2}\rceil +\lfloor  \log_2 n \rfloor-1$ \ for \ $n \geq 5$.
\end{corollary}
\begin{proof}
For $n=5$ or $6$ it follows from Lemma~\ref{lf}.
If the statement holds for an even number $n$, then it holds for $n+1$ as follows 
$$\begin{array}{lll}
\sa{Q_{n+1}} &=\sa{Q_n \square K_2} \\
& \leq \strut \sa{Q_n}+1  & \\
& \leq  \Big\lceil \dfrac{\strut n }{\strut 2} \Big\rceil +\lfloor \log_2 n \rfloor -1+1 & \\
& = \Big\lceil \dfrac{\strut n +1 }{\strut 2} \Big\rceil +\lfloor \log_2 (n+1) \rfloor -1 & ({\rm Since \ for \ even} \  n, \  
\lfloor \log_2 (n+1) \rfloor =\lfloor \log_2 n \rfloor).
\end{array}$$
So it suffices to prove the corollary for even number $n$.

It is easy to see that $n$ can be represented as a sum of $\lfloor \log_2 n \rfloor$ numbers of the form $2^k-2$, i.e. $n=\sum_{j=1}^m(2^{i_j}-2)$, $m \leq \lfloor \log_2n \rfloor$. In Theorem~\ref{tUp}, we represented $n=\sum_{j=1}^l(2^{i_j}-2)+r$. So $l \leq \lfloor \log_2 n \rfloor - \lfloor  \log_2 r \rfloor$.
As in Theorem~\ref{tUp}, 
$\sa{Q_n} \leq {n-r \over 2}+ \lfloor  \log_2 n \rfloor - \lfloor  \log_2 r \rfloor +\sa{Q_r}$, $r \in {\cal R}$.

Now if $r=2^k+2$, then we have
\\ \\
$
\begin{array}{ll}
 \sa{Q_n}  &\leq \dfrac{n-2^k-2 }{2}+ \lfloor  \log_2 n \rfloor -k+2^{k-1}+3 \\
& = \dfrac{\strut n}{\strut 2} - 2^{k-1}-1 + \lfloor  \log_2 n \rfloor -k + 2^{k-1}+3 \\
&= \dfrac{\strut n}{\strut 2}+ \lfloor  \log_2 n \rfloor -k +2.\end{array}
$
\\ \\
If $r=2^k+2^j-4$ and $k \geq j \geq 2$,
\\ \\
$
\begin{array}{ll}
 \sa{Q_n}  &\leq \dfrac{n-2^k-2^j+4 }{ 2}+ \lfloor  \log_2 n \rfloor -k+2^{k-1}+2^{j-1} \\
 & = \dfrac{\strut n }{\strut 2} - 2^{k-1}-2^{j-1}+2 + \lfloor  \log_2 n \rfloor -k + 2^{k-1}+2^{j-1} \\
&=\dfrac{\strut n }{\strut 2}+ \lfloor  \log_2 n \rfloor -k +2.
\end{array}
$
\\ \\
As we have seen in both cases $\sa{Q_n} \leq {n \over 2}+ \lfloor  \log_2 n \rfloor -k +2$. In both cases $k$ can be considered greater than or equal to $3$. As an example for $r=2^k+2^j-4$, assume $k=2$, then $j=2$ and $r=4$. Hence the last two numbers in $\sum_{j=1}^l(2^{i_j}-2)+r$, are $2^{i_l}-2$ and $r$ where $i_l \geq 3$. So we have $2^{i_l}-2+r=2^{i_l}-2+4=2^{i_l}+2$ which is a contradiction the choice of $r$.
Therefore $\sa{Q_n} \leq {n \over 2}+ \lfloor  \log_2 n \rfloor -k +2 \leq {n \over 2}+ \lfloor  \log_2 n \rfloor-1$. 
\end{proof} 
\section{Coloring and star arboricity}
The connection of star arboricity with other colorings such as incidence coloring and acyclic coloring are studied (see
 \cite{MR1428581} and \cite{MR1375101}).  In this section we consider the connection between square coloring and star arboricity of graphs. 

Square of a graph $G$ is a graph denoted by $G^2$ with ${\rm V}(G)={\rm V}(G^2)$ and two vertices are adjacent if their distance in $G$ is  at most $2$. A square-coloring of $G$ is a proper coloring of $G^2$. Let $\chi(G^2)$ be the minimum number of colors used in any square-coloring of $G$. 
\begin{theorem} \label{tsquare}
If $\chi(G^2) \leq k$ \ then \ $\sa{G}\leq \lceil {k \over 2}\rceil+1$, $k \geq 4$.
\end{theorem}
\begin{proof}
Let $c$ be a proper $k$-coloring of $G^2$ with color classes $C_1, C_2, \ldots, C_k$. We show that the degree of vertices in any induced subgraph on each pair of color classes is at most $1$. Assume to the contrary that there are two  classes $C_i$ and $C_j$, $1 \leq i, j \leq k$, such that an induced subgraph on them has a vertex of degree at least $2$. Without loss of generality let $v$ be a vertex in $C_i$ with ${\rm deg}_{\langle C_i \cup C_j\rangle }v=2$. So $v$ has at least two neighbors $u$ and $w$ in $C_j$. But in $G^2$, $u$ and $w$ are adjacent, that is contradiction with $c$ being a proper coloring. Thus the vertices of $G$ are partitioned into $k$ independent sets such that an induced subgraph on each pair of them is a matching. 

Now using $G$ we construct a graph $H$ as follows. Each vertex of $H$ corresponds to a color class of $G$ and two vertices are adjacent if there is an edge between their corresponding color classes. Clearly $H$ is a subgraph of $K_k$. Thus from Theorem~\ref{tcomplete}, $\sa{H} \leq  \lceil {k \over 2} \rceil+1$. By a blow up each galaxy of $H$ can be mapped to a galaxy of $G$.
\end{proof}
%
Note that the result in the theorem can be sharp. As an example for $Q_{2^t-1}$, we have $\chi(Q_{2^t-1}^2)=2^t$ (see \cite{MR2452764} and \cite{MR2098840}) which implies that $\sa{Q_{2^t-1}} \leq 2^{t-1}+1$.
\section*{Acknowledgements}
The authors appreciate Maryam Sharifzadeh for her useful comments and  presence in the earlier discussions and Hamed Sarvari for his computer program.

%
\end{document}